\documentclass{patmorin}
\usepackage{pat,float,subcaption}
\usepackage[T1]{fontenc}
\usepackage[utf8]{inputenc}
\usepackage{mathtools}
\usepackage{thmtools}
\usepackage{thm-restate}
\usepackage{bm}
\usepackage{stmaryrd}
\newcommand{\sm}{\smallsetminus}
\usepackage[inline]{enumitem}
\usepackage{soul}
\usepackage[normalem]{ulem}

\crefname{p}{}{}
\creflabelformat{p}{#2(#1)#3}

\usepackage{todonotes}
\newenvironment{clmproof}{\begin{proof}[Proof of Claim:]}{\end{proof}}

\usepackage[longnamesfirst,numbers,sort&compress]{natbib}

\newcommand{\ZZ}{\mathbb{Z}}
\newcommand{\Acal}{\mathcal{A}}

\newcommand{\RR}{\mathbb R}

\newcommand{\SSS}{\mathbb{S}}

\newcommand{\calV}{\mathcal{V}}

\DeclareMathOperator{\init}{init}
\DeclareMathOperator{\ter}{ter}

\DeclareMathOperator{\xx}{\mathbf{x}}
\DeclareMathOperator{\yy}{\mathbf{y}}
\DeclareMathOperator{\supp}{\mathsf{supp}}

\definecolor{OliveGreen}{RGB}{120, 140, 47}

\title{\MakeUppercase{$2$-dimensional unit vector flows}%
  \thanks{This research was partly funded by NSERC.}}

\author{Hussein Houdrouge%
  \thanks{School of Computer Science, Carleton University.},\quad
  Bobby Miraftab\footnotemark[2],\quad
  and Pat Morin\footnotemark[2]}

\date{}

\begin{document}
\maketitle

\begin{abstract}
We study \emph{$2$-dimensional unit vector flows} on graphs, that is, nowhere-zero flows that assign to each oriented edge a unit vector in $\RR^{3}$.
We give a new geometric characterization of $\SSS^{2}$-flows on cubic graphs.
We also prove that the class of cubic graphs admitting an $\SSS^{2}$-flow is closed under a natural composition operation, which yields further constructions; in particular, \emph{blowing up} a vertex into a triangle preserves the existence of an $\SSS^{2}$-flow.
Our second contribution is algebraic: we extend the rank-based approach of [SIAM J. Discrete Math., 29 (2015), pp.~2166--2178] from $\SSS^{1}$-flows to $\SSS^{2}$-flows.
More precisely, we show that if an $\SSS^{2}$-flow $\varphi$ satisfies $\operatorname{rank}(S_{\mathbb{Q}}(\varphi))\le 2$ and $S_{\mathbb{Q}}(\varphi)$ is odd-coordinate-free, then the graph admits a nowhere-zero $4$-flow.
\end{abstract}

\section{Introduction}
In this paper, an oriented graph $\vec{G}$ is an undirected graph $G = (V, E)$ with two maps $\init\colon E \longrightarrow V$ and $\ter\colon E \longrightarrow V$ that assign an initial and a terminal vertex for every edge $e \in E(G)$. The set $\mathdefin{E(\vec{G})}$ denotes the set of oriented edges, that is $E(\vec{G}) \coloneqq  \{ (e, \init(e), \ter(e)) \; |\; e \in E(G)\}$. 
We will simply refer to an element $(e, \init(e), \ter(e)) \in E(\vec{G})$ by $e$. 
Sometimes, we refer to an oriented graph $\vec{G}$ as an \defin{orientation} of the graph $G$. 
For $X \subseteq V(G)$, we define $ E^+(X)$ to be the set of all edges $e$ whose $\init(e) \in X$ and $\ter(e) \in V(G)\sm X$. 
The set $E^-(X) $ is defined as $ E^+(V(G) \sm X)$. 
For a vertex $v \in V(G)$, we write $E^+(v)$ instead of $E^+(\{v\})$ and $ E^-(v)$ instead of $E^-(\{v\})$.

Let $ \Acal$ be an additive abelian group, we define a \defin{$\Acal$-circulation} over an oriented graph $\vec{G}$ as function $ \varphi \colon E(\vec{G}) \rightarrow \Acal $ that satisfies Kirchhoff's current Law (KCL):
\begin{equation}\label{KL}
\sum_{e \in E^{+}(v)} \varphi(e) = \sum_{e \in E^{-}(v)} \varphi(e)
\end{equation}
for every $v \in V(G)$. 
An \defin{$\Acal$-flow} for an oriented graph $\vec{G}$ is a nowhere zero $\Acal$-circulation $\varphi$ of $\vec{G}$, that is $\varphi(e) \neq 0$ for every $e \in E(\vec{G})$. Note that setting $\varphi(e) = -\varphi(e)$ after swapping $\init(e)$ and $\ter(e)$ will preserve \cref{KL}. 
Thus, if an oriented graph $\vec{G}$ has an $\Acal$-flow, then any oriented graph with $G$ as underling graph has an $\Acal$-flow. Therefore, an undirected graph $G$ has an $\Acal$-flow if and only if one of its orientations $\vec{G}$ has an $\Acal$-flow.

For a positive integer $k$, a $\ZZ$-flow $\varphi$ for a graph $G$ that satisfies $0 < |\varphi(e)| < k$ for every $e \in E(\vec{G})$ for an orientation $\vec{G}$ is called a \defin{$k$-flow}. In $1950$, Tutte proved that having a $k$-flow is equivalent to having a $\mathbb Z_k$-flow, see \cite[Theorem 6.3.3]{diestel2017graph}. 
In \cite{DBLP:journals/jct/Seymour81a}, Seymour proves that every bridgeless graph has a $6$-flow.
Later, alternative proofs were discovered by \citet{new_6_flow} and \citet{new_6_flow_2}. However, the characterisation of graphs with $k$-flows for $k<6$ remains open. For $k = 5$, Tutte conjectured the following.

\begin{conj}[Tutte's $5$-flow conjecture]\label{tutte_five}
Every bridgeless graph has a $5$-flow.
\end{conj}

In the definition of $\Acal$-flow, we can replace the abelian group $\Acal$ by a vector space $\mathbb{V}$ to define a \defin{vector flow}. Since $\mathbb{V}$ is an abelian group with respect to the addition operation with identity element being the zero vector, denoted by $\mathbf{0}$, vector flow is well defined. In this work, we are interested in vector spaces over $\mathbb{R}$ equipped with the Euclidean norm $\|\cdot\|$. Let $ r \ge 2 $ be a real number and $d$ a positive integer. 
An \defin{$(r, d)$-flow} for a graph $G$ is an $\mathbb{R}^d$-flow such that $\|\varphi(e)\|$ lies in the interval $[1, r - 1]$ for every edge $e \in E(\vec{G})$ for any orientation $\vec{G}$. For $r = 2$, the $(2, d)$-flow is called \defin{unit vector flow}. Since every element in $\SSS^{d - 1}$, the unit sphere of dimension $d - 1$, is a unit vector in $\R^{d}$, we refer to $(2, d)$-flow by \defin{$\SSS^{d - 1}$-flow}. 

In order to approach \cref{tutte_five}, \citet{bobby_conjecture} proposes the following two conjectures that imply \cref{tutte_five}.  
\begin{conj}\label{Bobby_conj}
Every bridgeless cubic graph has an $\SSS^2$-flow.
\end{conj}
\begin{conj}\label{Bobby_complement_conj}
There exists a map $q:\SSS^2 \rightarrow \{\pm 1, \pm 2,  \pm 3, \pm 4\} $ such that the antipodal points of $ \SSS^2 $ receive opposite values, and any three points which are equidistant on a great circle have values which sum to zero.
\end{conj}

Vector flows and in particular  \Cref{Bobby_conj} has been studied by several authors, \cite{VecAndInt,thom2,NZF}.
In this paper, we study $\SSS^2$-flow from both geometric and algebraic perspectives.
More precisely in order to prove or disprove \cref{Bobby_conj}, we provide a new geometric characterisation for $\SSS^2$-flows. We show that a cubic graph $G$ has an $\SSS^2$-flow is equivalent to having \defin{equiangular $\SSS^2$-immersion}. Informally, an equiangular $\SSS^2$-immersion is a mapping of $V(G)$ to points in $\SSS^2$, and a mapping of $E(G)$ to arcs of great circles. The corresponding arcs of edges that are incident to the same vertex are at angle of $2\pi/3$ apart from each other. Precisely, we prove the following.

\begin{thm}\label{thm:S2}
A cubic graph $ G $ admits an equiangular $ \SSS^2 $-immersion if and only if $ G $ admits an $ \mathbb{S}^2 $-flow.
\end{thm}

In addition to this characterisation, we illustrate the use of \cref{thm:S2}. We provide an equiangular $\SSS^2$-immersion for some families including Petersen graph and for generalised variations of Petersen graph. 
In \cref{thm::bipartite_immersion}, we show that bipartite graphs are exactly the graphs that admit an equiangular $\SSS^2$-immersion to one or two points in $\SSS^2$. 
Finally, we show how to construct graphs with an $\SSS^2$-flow from two cubic graphs that have an $\SSS^2$-flow, see \cref{thm::injection_thm}.

Another interesting question in this area is the following:

\begin{problem}\label{problem1}
For a non-negative integer $d$, characterize $\SSS^d$-flows for which if $G$ admits an $\SSS^d$-flow with certain properties, then $G$ also admits a nowhere-zero integer $k$-flow for some $k \leq 5$.
\end{problem}

\citet{VecAndInt} studied the preceding question for $\SSS^1$-flow
\begin{lem}{\rm \cite[Theorem 1.10]{VecAndInt}}
If a graph $G$ admits a vector $\SSS^1$-ﬂow with rank at most two, then $G$ admits a nowhere-zero integer $3$-ﬂow.
\end{lem}

Here we explain the terminology of vector flows with rank.
Suppose $\calV := \{\varphi(e)\mid e\in E(\vec{G})\} = \{{\bm v}_{1},\dots, {\bm v}_{b}\}$. Then for every ${\bm v}_i \in \calV$, we let $E_i$ to be the set of edges with flow value equal to ${\bm v}_i$. For a vertex $v \in V(G)$, we set $\epsilon_i(v)\coloneqq  |E^+(v) \cap E_i| - |E^-(v) \cap E_i|$. As in \cite{VecAndInt}, we define the \defin{balanced vector} $\bm{\epsilon}(v)$ of a vertex $v \in V(G)$ as $\boldsymbol{\epsilon}(v) = \left( \epsilon_1(v), \dots, \epsilon_b(v) \right)$. We define  $\mathdefin{S_{\mathbb{Q}}(\varphi)}$ as the span of $\{\bm{\epsilon}(v)| v \in V(G)\}$ over the field of rational numbers $\mathbb{Q}$. We say $S_{\mathbb{Q}}(\varphi)$ is \defin{odd-coordinate-free} if it contains no integer vector with exactly one odd coordinate. 

We now are ready to present the extension of \cite[Theorem 1.10]{VecAndInt} to $\SSS^2$-flows.

\begin{thm}\label{rank_odd_cor}
If a graph $G$ admits an $\SSS^2$-flow such that $S_{\mathbb{Q}}(\varphi)$ is odd-coordinate-free of rank at most $2$, then $G$ admits a $4$-flow.
\end{thm}


\section{Preliminaries}\label{prelim}
First, we start with recalling some basic facts about $\SSS^0$-flows.
\begin{obs}\label{obs:S0}
For a graph $G$, the following statements are equivalent:
\begin{compactenum}
\item[{\rm (i)}] $G$ has an $\SSS^0$-flow.
\item[{\rm (ii)}] $G$ has a $\Z_2$-flow.
\item[{\rm (iii)}] All vertices of $G$ have even degree.
\end{compactenum}
\end{obs}
\citet{thom2} characterized when a cubic graph admits an $\SSS^1 $-flow. Let $ R_k $ denote the set of $k^{th}$ roots of unity, the set of complex numbers $z$ satisfying $z^k = 1$.
\begin{lem}[{\cite[Proposition 1]{thom2}}]\label{cubic}
Let $G$ be a graph. Then the following statements are equivalent:
\begin{compactenum}
\item[{\rm (i)}] $G$ has a $\ZZ_3$-flow.
\item[{\rm (ii)}] $G$ has an $R_3$-flow.
\end{compactenum}
Moreover, both {\rm (i)} and {\rm (ii)} imply:
\begin{compactenum}
    \item[{\rm (iii)}] $G$ has an $\SSS^1$-flow.
\end{compactenum}
If $G$ is cubic, then the three statements {\rm (i)}, {\rm (ii)}, and {\rm (iii)} are equivalent. In this case, $G$ satisfies these conditions if and only if it is bipartite.
\end{lem}

Note that most of our theorems such as \cref{thm:S2} are phrased in terms of cubic graphs. The reason we focus on cubic graphs is summarized in the following paragraph. Consider the following reduction rules on a graph $G$.
\begin{compactenum}
    
    \item[(R1)] For every $v \in V(G)$ with neighbours $w_1, \dots, w_{2k}$ for an integer $k > 1 $, replace $v$ with $v_1, \dots , v_{k}$ where each $v_i$ for $i \in \{1, \dots k\}$ has neighbours $w_{2i - 1}, w_{2i}$.
    \item[(R2)] For every vertex $v \in V(G)$ with neighbours $w_1, \ldots , w_{2k+3}$ where $k$ is a natural number, replace $v$ with vertices $v_1, \ldots, v_{k+1}$ where $v_i$ has neighbours $w_{2i - 1}$ and $w_{2i}$ for each $i \in \{1, \ldots, k\}$ and $v_{k+1}$ has neighbours $w_{2k+1}$, $w_{2k+2}$, and $w_{2k+3}$ see \Cref{fig:split_vertex}.
    \item[(R3)] For every $v \in V(G)$ of degree $2$, suppress $v$ (replacing $v$ and its two incident edges with a single edge).
\end{compactenum}

\begin{figure}[H]
    \centering
    \includegraphics[scale=0.7]{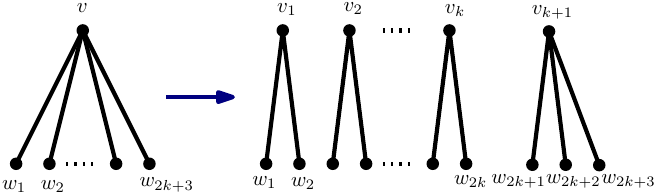}
    \caption{Splitting a vertex of odd degree.}
    \label{fig:split_vertex}
\end{figure}

\begin{lem}
Let $G$ be a graph and let $G'$ be the graph obtained from applying $(R1)$, $(R2)$, and $(R3)$ extensively on $G$, then If $G'$ has an $\Acal$-flow, then $G$ has an $\Acal$-flow. 
\end{lem}

It is also important to mention the following corollary for Tutte's flow polynomial.
\begin{proposition}[\cite{diestel2017graph}, Corollary 6.3.2]\label{group_order}
    Given two finite abelian groups $\Acal$ and $\Acal'$ with the same order. Then, a graph $G$ has an $\Acal$-flow if and only if $G$ has an $\Acal'$-flow. 
\end{proposition}

Now, we make precise what is discussed in the introduction regarding reversing an edge orientation:
\begin{obs}\label{reversing}
    Let $\vec{G}$ be an oriented graph, let $\varphi\colon E(\vec{G})\to\Acal$ for some abelian group $\Acal$, let $\vec{G}'$ be obtained from $\vec{G}$ by reversing the orientation of some edge $e$ and let $\varphi'\colon E(\vec{G})\to\Acal$ be obtained from $\varphi$ by setting $\varphi'(e)=-\varphi(e)$.  Then $\varphi$ is a $\Acal$-flow for $\vec{G}$ if and only if $\varphi'$ is a $\Acal$-flow for $\vec{G}'$. 
\end{obs}

Finally, we recall the following proposition about the net flow across any edge cut. 
\begin{lem}{\rm\cite[Proposition 6.1.1]{diestel2017graph}}\label{lem:cut}
Let $\vec{G}$ be an oriented graph with an $\Acal$-circulation $\varphi$ for an abelian group $\Acal$.
Then for any $X \subseteq V(G)$,
$$\sum_{e\in E^+(X)} \varphi(e)= \sum_{e\in E^-(X)}\varphi(e).$$
\end{lem}
The last lemma implies that any graph with a bridge does not have an $\Acal$-flow. 



\section{A Geometric Approach of \texorpdfstring{$\SSS^2$}{S2}-Flows}\label{proof_of_geometric_theorem}

\begin{obs}\label{coplanar}
    Let $\mathbf{v}_1$, $\mathbf{v}_2$ and $\mathbf{v}_3$ be three unit vectors in $\RR^3$.  Then the following two conditions are equivalent:
    \begin{enumerate}[nosep,noitemsep]
      \item $\sum_{i=1}^3\mathbf{v}_i=\mathbf{0}$; and
      \item the (point set) $\{\mathbf{0},\mathbf{v}_1,\mathbf{v}_2,\mathbf{v}_3\}$ is coplanar and the angle between $\mathbf{v}_i$ and $\mathbf{v}_j$ is $2\pi/3$ for each $1\le i<j\le 3$.
  \end{enumerate}
\end{obs}

In the following, we treat points of $\SSS^2$ as unit vectors in $\RR^3$ and use $\mathbf{0}:=(0,0,0)$ to denote the origin in $\RR^3$.  
For each plane $\pi$ in $\RR^3$ that contains the origin, $C_{\pi}:=\pi\cap \SSS^2$ is called a \defin{great circle}.
Any closed non-empty connected subset $A\subseteq C_{\pi}$ is called a \defin{geodesic arc} (in $\SSS^2$).  Note that, under this definition, $C_{\pi}$ is itself a geodesic arc.  A geodesic arc $A$ is \defin{proper} if it a strict subset of some great circle.  

A \defin{simple curve} $A:[0,1]\to\SSS^2$ in $\SSS^2$ is a continuous function with the property that $A(x)\neq A(y)$ for each $0\le x < y<1$.  For a simple curve $A$ and real numbers $0\le x<y\le 1$, we use notation $\mathdefin{A^{[x,y]}}:=\{A(t):x\le t\le y\}$.

A \defin{directed} geodesic arc is a simple curve $A$ in $\SSS^2$ with the property that its image $A^{[0,1]}$ is a geodesic arc parameterized so that the length of $A^{[0,s]}$ is equal to $s$ times the length of $A^{[0,1]}$.
For a directed geodesic arc $A$, the \defin{completion} of $A$ is the directed geodesic arc $C:[0,1]\to\SSS^2$ whose image is a great circle parameterized so that $A^{[0,1]}=\{C(t):0\le t\le \ell/2\pi\}$, where $\ell$ is the length of $A$. 
In other words, the curve $C$ traverses the great circle $C^{[0,1]}$ that contains $A^{[0,1]}$ beginning and ending at $A(0)$ and travelling in the same direction as $A$.  In most of what follows we will usually not distinguish between a directed geodesic arc $A$ and the (undirected) geodesic arc $A^{[0,1]}$.


Let $\vec{G}$ be an orientation of a graph $G$. An \defin{$\SSS^2 $-immersion} of $ \vec{G} $ is a function $\gamma$ with domain $V(G)\cup E(\vec{G})$ satisfying the following conditions:
\begin{enumerate}[label={\rm(\roman*)}]
 \item\label{vertex_to_sphere} For each $v\in V(G)$, $\gamma(v)\in\SSS^2$. (The vertices of $ G $ are mapped to points of $\SSS^2$)
 \item\label{edge_to_geodesic} For each $vw\in E(\vec{G})$, $\gamma(vw)$ is a directed geodesic arc $A_{vw}$ with $A_{vw}(0)=\gamma(v)$ and $A_{vw}(1)=\gamma(w)$. (Each directed edge of $\vec{G}$ is mapped to a directed geodesic arc in $\SSS^2$ that begins $\gamma(v)$ and ends at $\gamma(w)$.
\end{enumerate}

It is worth noting some differences between $\SSS^2$-immersions and embeddings of a graph $G$ in $\SSS^2$.  Unlike an embedding, there is no requirement that $\gamma:V(G)\to\SSS^2$. An $\SSS^2$-immersion $\gamma$ may have $\gamma(v)=\gamma(w)$ for distinct $v,w\in V(G)$. Furthermore, unlike a crossing-free embedding, there are no restrictions on the images of edges in an immersion. An $\SSS^2$-immersion $\gamma$ may have  $\gamma(vw)\cap \gamma(xy)\neq\emptyset$ and even $\gamma(vw)\subseteq \gamma(xy)$ for disjoint $vw,xy\in E(\vec{G})$.  


We now introduce a notation that allows us to discuss the spherical angle between the images of two edges incident to a common vertex $v$ in an $\SSS^2$-immersion.
For an $\SSS^2$-immersion $\gamma$ of $\vec{G}$, and a directed edge $vw\in E(\vec{G})$ with $A_{vw}:=\gamma(vw)$, we define $\gamma^{(1/3)}(vw):=A_{vw}^{[0,1/3]}$ and $\gamma^{(1/3)}(wv):=A_{vw}^{[2/3,1]}$.\footnote{The constants $1/3$ and $2/3$ are not critical here. The values $1/2\pm \epsilon$ for any $\epsilon>0$ would be sufficient.} Note that each of $\gamma^{(1/3)}(vw)$ and $\gamma^{(1/3)}(wv)$ is a proper geodesic arc, even if $\gamma(vw)$ is not.  
For an orientation $\vec{G}$ of a cubic graph $G$, an $\SSS^2$-immersion $\gamma$ of $\vec{G}$ is \defin{equiangular} if, 
\begin{enumerate}[label={\rm(\roman*)}]
  \setcounter{enumi}{2}
  \item\label{one_twenty} for each $v\in V(G)$ and each pair of undirected edges $vx$ and $vy$ incident to $v$, the spherical angle between $\gamma^{(1/3)}(vx)$ and $\gamma^{(1/3)}(vy)$ is exactly $2\pi/3$.
\end{enumerate}

We have the following observation about $\SSS^2$-immersions, akin to \cref{reversing}:

\begin{obs}\label{reversing_immersions}
    Let $\vec{G}$ be an oriented graph, let $\gamma$ be an equiangular $\SSS^2$-immersion of $G$, let $\vec{G}'$ be obtained from $\vec{G}$ by reversing the orientation of some edge $e$ and let $\gamma'$ be obtained from $\gamma$ by setting $\gamma'(e)(t):=\gamma(e)(1-t)$, for each $0\le t\le 1$.  Then $\gamma$ is an equiangular $\SSS^2$-immersion of $\vec{G}$ if and only if $\gamma'$ is an equiangular $\SSS^2$-immersion of $\vec{G}'$. 
\end{obs}

\Cref{reversing_immersions} implies that if some orientation $\vec{G}$ of a cubic graph $G$ has an equiangular $\SSS^2$-immersion then every orientation of $G$ does.  Thus, having equiangular $\SSS^2$-immersions is a property of (unoriented) cubic graphs. The following theorem shows that this property completely characterizes cubic graphs that have $\SSS^2$-flows.
 

\begin{mythm}{\ref{thm:S2}}
A cubic graph $ G $ admits an equiangular $ \SSS^2 $-immersion if and only if $ G $ admits an $ \mathbb{S}^2 $-flow.
\end{mythm}

\begin{proof}
For the forward implication, let $\gamma$ be an $\SSS^2$-immersion of $G$ and fix an arbitrary orientation $\vec{G}$ of  $G$. Consider some oriented edge $vw\in E(\vec{G})$, let $A_{vw}:=\gamma(vw)$ and let $C_{vw}$ be the completion of $A_{vw}$.    
The great circle $C_{vw}$ is contained in a plane $\pi_{vw}\subset\RR^3$, and there are exactly two unit vectors $\mathbf{v}$ and $-\mathbf{v}$ orthogonal to $\pi_{vw}$.  If $C_{vw}$ winds counterclockwise around $\mathbf{v}$ then we define $\varphi(vw):=\mathbf{v}$, otherwise we define $\varphi(vw):=-\mathbf{v}$.  In either case, $C_{vw}$, winds counterclockwise around $\pmb{\varphi}(vw)$.

By construction $\varphi$ is a map from $E(\vec{G})$ to $\SSS^2$.  We now argue that $\varphi$ is an $\SSS^2$-flow for $\vec{G}$.  
Consider some vertex $v$ of $G$.  
By \cref{reversing,reversing_immersions} we may assume, without loss of generality, that the three edges $vv_1$, $vv_2$, and $vv_3$ are oriented away from $v$.  
We must therefore show that $\sum_{i=1}^3 \pmb{\varphi}(vv_i)=\mathbf{0}$. 
For each $i\in\{1,2,3\}$, the vector $\pmb{\varphi}(vv_i)$ is orthogonal to the vector $\gamma(v)$, so the vectors $\pmb{\varphi}(vv_1)$, $\pmb{\varphi}(vv_2)$, and $\pmb{\varphi}(vv_3)$ are contained in a plane $\pi_v$ that also contains the origin.  By the definition of $\pmb{\varphi}(vv_i)$, $\pi_v$ is orthogonal to the vector $\pmb{\gamma}(v)$.  
For each $i\in\{1,2,3\}$, let $A_i:=\gamma(vv_i)$ and let $$\mathbf{v}_i:=\lim_{\epsilon\to 0}\frac{A_i(\epsilon)-A_i(0)}{\|A_i(\epsilon)-A_i(0)\|}.$$  
In words, $\mathbf{v}_i$ is the unit vector that is parallel to the direction in which $A_i$ departs from $\gamma(v)$.  The spherical angle between $\gamma^{(1/3)}(vv_i)$ and $\gamma^{(1/3)}(vv_j)$ is equal to the angle between the two vectors $\mathbf{v}_i$ and $\mathbf{v}_j$, for each $1\le i<j\le 3$.  
For each $i\in\{1,2,3\}$, the vector $\mathbf{v}_i$ is contained in the plane $\pi_v$ and is, in fact, obtained by rotating $\pmb{\varphi}(vv_i)$ clockwise around $\pmb{\gamma}(v)$ by an angle of $\pi/2$. 
Therefore, the angle between $\pmb{\varphi}(vv_i)$ and $\pmb{\varphi}(vv_j)$ is equal to the angle between $\mathbf{v_i}$ and $\mathbf{v}_j$. Since $\gamma$ is equiangular, this angle is $2\pi /3$.  Therefore $\pmb{\varphi}(vv_1)$, $\pmb{\varphi}(vv_1)$, and $\pmb{\varphi}(vv_2)$ are coplanar with the origin and the angle between $\pmb{\varphi}(vv_i)$, and $\pmb{\varphi}(vv_j)$ is $2\pi/3$, for each $1\le i<j\le 3$.  Therefore, \cref{coplanar} implies that $\sum_{i=1}^{3}\pmb{\varphi}(vv_i)=\mathbf{0}$, as required.

To establish the backward implication, fix an arbitrary orientation $\vec{G}$ of $G$ and let $\varphi$ be an $\SSS^2$-flow of $\vec{G}$.  For each vertex $v$ of $G$ arbitrarily select one of the two possible cyclic orders on the three edges of $G$ incident to $v$. We use $v_1,v_2,v_3$ to denote the neighbours of $v$, using this cyclic order.  Now consider an arbitrary vertex $ v $ of $ G $. 
By \cref{reversing} we may assume, without loss of generality, that the three edges $vv_1$, $vv_2$, and $vv_3$ are oriented away from $v$.  Define $\pmb{\gamma}(v)\coloneqq(2/\sqrt{3})(\pmb{\varphi}(vv_1)\times\pmb{\varphi}(vv_2))$.  
By \cref{coplanar}, the angle between the vectors $\pmb{\varphi}(vv_1)$ and $\pmb{\varphi}(vv_2)$ is $2\pi/3$, so $\gamma(v)$ is a unit vector (since $\sin(2\pi/3)=\sqrt{3}/2$).  Also by \cref{coplanar}, $v_1$, $v_2$, and $v_3$ are coplanar with the origin, so $\pmb{\gamma}(v)=2/\sqrt{3}\pmb{\varphi}(vv_2)\times\pmb{\varphi}(vv_3)=2/\sqrt{3}\pmb{\varphi}(vv_3)\times\pmb{\varphi}(vv_1)$ 
This defines $\gamma(v)$ for each vertex $v$ of $G$. 

What remains is to define $\gamma(uv)$ for each edge $uv\in\E(\vec{G})$. Without loss of generality, assume that $u=v_1$ and that $v=u_1$.  
By \cref{reversing} we may assume, without loss of generality, that $uu_2,uu_3,uv,vv_2,vv_3$ are the five (oriented) edges of $\vec{G}$ incident to $u$ and $v$.  
Then $\pmb{\gamma}(u)=2/\sqrt{3}(\pmb{\varphi}(uv)\times \pmb{\varphi}(uu_2))$ and $\gamma(v)=-2/\sqrt{3}(\varphi(uv)\times\varphi(vv_2))$. 
Geometrically, this means that $\gamma(u)$ and $\gamma(v)$ are each contained in the plane $\pi_{uv}$ that contains the origin and is orthogonal to $\varphi(uv)$. 
Let $C_{uv}:=\pi_{uv}\cap \SSS^2$, so $C_{uv}$ is a great circle that contains $\gamma(u)$ and $\gamma(v)$. 
We define $\gamma(uv)$ to the directed arc $A_{uv}$ that is contained in $C_{uv}$, has $A_{uv}(0)=\gamma(u)$, $A_{uv}(1)=\gamma(v)$ and whose extension winds counterclockwise around $\pmb{\gamma}(uv)$.


We now argue that $\gamma$ is an $\SSS^2$-immersion of $G$.  By definition, $\gamma(v)\in\SSS^2$ for each $v\in V(G)$, so $\gamma$ satisfies Condition~\ref{vertex_to_sphere}.  By definition, each edge $uv$ of $\vec{G}$ is mapped to a directed geodesic arc $A_{uv}$ (contained in $C_{uv}$) with $A_{uv}(0)=\gamma(u)$ and $A_{uv}(1)=\gamma(v)$, as required by Condition~\ref{edge_to_geodesic}.  Therefore $\gamma$ is an $\SSS^2$-immersion.  All that remains is to show that $\gamma$ is equiangular.

Let $v$ be an arbitrarys vertex of $G$.  By \cref{reversing} we may assume without loss of generality, that $vv_1$, $vv_2$, and $vv_3$ are the three oriented edges of $\vec{G}$ incident to $v$.  Since $\gamma$ is a $\SSS^2$-flow, $\sum_{i=1}^3 \varphi(vv_i)=\mathbf{0}$.  By \cref{coplanar}, there is a plane $\pi_v$ that contains $$\{\mathbf{0},\varphi(vv_1),\varphi(vv_2),\varphi(vv_3)\}.$$  By the choice of $\gamma(v)$, this plane is orthogonal to the vector $\pmb{\gamma}(v)$.
For each $i\in\{1,2,3\}$, let $A_i:=\gamma(vv_i)$ and let $C_i$ be the completion of $A_i$.
Then, for each $i\in\{1,2,3\}$, the tangent vector $$\mathbf{v}_i:=\lim_{\epsilon\mathbin{\downarrow} 0} \frac{(C_i(\epsilon)-C_i(0))}{\|C_i(\epsilon)-C_i(0)\|}$$ is contained in the plane $\pi_v$ and is, in fact, obtained by rotating $\pmb{\varphi}(vv_i)$ clockwise around $\pmb{\gamma}(v)$ by an angle of $\pi/2$. 
Therefore, the angle between $\mathbf{v}_i$ and $\mathbf{v}_j$ is equal to the angle between $\pmb{\varphi}(vv_i)$ and $\pmb{\varphi}(vv_j)$, for each $1\le i<j\le 3$.  
By \cref{coplanar} this latter angle is $2\pi/3$. Therefore the spherical angle between $\gamma^{(1/3)}(vv_i)$ and $\gamma^{(1/3)}(vv_j)$ (which is equal to the angle between $\mathbf{v}_i$ and $\mathbf{v}_j$) is $2\pi/3$.  Therefore $\gamma$ satisfies Condition~\ref{one_twenty}, so $\gamma$ is an equiangular $\SSS^2$-immersion of $\vec{G}$. 
\end{proof}


\section{Examples of Graphs with an \texorpdfstring{$\SSS^2$}{S2}-Flows}\label{examples}
In this section, we illustrate the use of \cref{thm:S2} to show that some graphs admits an $\SSS^2$-flow. An important tools to accomplish our objective is \defin{the Intermediate Value Theorem}. Thus, we begin by recalling this fact about continuous functions. A space $S$ is \defin{path-connected}, if for every pair of points $x$ and $y$ in $S$, there exists a continuous function  
\begin{equation}
\gamma \colon [0,1] \to S \quad \text{such that} \quad \gamma(0) = x \quad \text{and} \quad \gamma(1) = y.
\end{equation}
\begin{lem}\label{inter_mediate_value_thm}[Intermediate Value Theorem]
Let $S$ be a path-connected space, and let $ f \colon S \to \RR $ be a continuous function. If $ a, b \in S$ then $f$ attains all the values between $f(a)$ and $f(b)$.
\end{lem}
The next lemma describes a local modification that preserves the immersion.
\begin{lem}\label{lem:upper_half}
Let $ G $ be a graph with an $\SSS^2 $-immersion. Then, for any vertex $ v \in V(G) $, we can replace $ v $ with the antipodal point $ -v $ and modify only the edges incident to $ v $ to obtain another $\SSS^2 $-immersion of $ G $ that preserves the angles formed by any pair of edges meeting at a common vertex.
\end{lem}
\begin{proof}
The great circles that contain the edges incident to $ v $ intersect both at $ v $ and at $ -v $. 
For each such edge $ vw $, we replace it with the arc of the same great circle that connects $ -v $ to $ w $, choosing the arc that contains the original edge $ vw $.
To see that this transformation preserves the angles between edges incident to $ v $, observe that this can be visualized as moving the point $ v $ along the great circle, away from $ w $, until it reaches the antipodal point $ -v $. This operation effectively reverses the direction in which each edge $ vw $ departs from $ v $. Since the direction of each incident edge is negated, the angles between any pair of them remain unchanged.
This operation does not affect the direction of any edge at its other endpoint. 
In particular, for each neighbor $ w $ of $ v $, the direction in which the edge $ wv $ departs from $ w $ remains unchanged.
Therefore, the resulting immersion preserves all local geometric properties at each vertex, and in particular, all angles at the vertices remain the same.
\end{proof}

\subsection{Generalised Peterson Graph}
Let $a,b,p\in\mathbb{N}$ with $\left\lceil \tfrac{p}{6}\right\rceil \le a,b \le \left\lfloor \tfrac{p}{2}\right\rfloor$. 
We define the \defin{quasi-Petersen graph} $G_{a,b,p}$ with parameters $a,b,p$, as follows.
\begin{compactenum}
    \item The vertex set is the disjoint union of
    \[
        V\coloneqq \{v_0,\dots,v_{p-1}\}\quad\text{and}\quad
        W\coloneqq \{w_0,\dots,w_{p-1}\},
    \]
    \item The edge set $E(G_{a,b,p}) = C_V \cup C_W \cup M$, where (all subscripts are taken modulo $p$, and $[p] := \{0,1,\dots, p - 1\}$) 
    \[
    C_V\coloneqq\{\, v_i v_{i+a} : i\in [p] \,\},\qquad
    C_W\coloneqq \{\, w_i w_{i+b} : i\in [p] \,\},\qquad
    M\coloneqq \{\, v_i w_i : i\in [p] \,\}.
\]
\end{compactenum}

\begin{obs}
The following holds true:
\begin{enumerate}
\item One can verify that $G_{1,2,5}$ is the Petersen graph. 
More precisely if $\gcd(a,p)=1$ then $C_V$ is a cycle. 
Similarly, $\gcd(b,p)=1$ implies that $C_W$ is a cycle. 
In particular, for any prime $p$, taking $a=\floor{p/2}-1$ and $b=\floor{p/2}$ yields a graph $G_{a,b,p}$ consisting of two cycles $C_V$ and $C_W$ with a matching $M$ between them. 
\item For any odd $p\ge 5$, the choices $a=\floor{p/2}-1$ and $b=\floor{p/2}$ satisfy $\,\tfrac{p}{6}<a,b<\tfrac{p}{2}$. 
\item For any even $p\ge 8$, the choices $a=p/2-2$ and $b=p/2-1$ satisfy this requirement.
\item The \defin{generlized Petersen graph}
${\displaystyle G(n,k)}$\footnote{Here we use the Watkins' notation. However, some authors use the notation 
${\displaystyle GPG(n,k)}$.} is a graph with vertex set
$${\displaystyle \{u_{0},u_{1},\ldots ,u_{n-1},v_{0},v_{1},\ldots ,v_{n-1}\}}$$ and edge set
${\displaystyle \{u_{i}u_{i+1},u_{i}v_{i},v_{i}v_{i+k}\mid 0\leq i\leq n-1\}}$ where subscripts are to be read modulo 
${\displaystyle n}$ and where 
${\displaystyle k<n/2}$. 
One can see that our notation generalizes the notion of a generalized Petersen graph, where $\gcd(a,p)=1$.
\end{enumerate}
\end{obs}

We say that an $\SSS^2$-immersion is \defin{injective} if the vertex map $\gamma\!\restriction_{V(G)}\colon V(G)\to \SSS^2$ is injective.

\begin{proposition}
For any positive integers $a,b,p$ with $\tfrac{p}{6}<a,b<\tfrac{p}{2}$, the graph $G_{a,b,p}$ has an injective equiangular $\SSS^2$-immersion.
\end{proposition}

\medskip
\noindent\textbf{Proof.}
Place the vertices of $V=V(G_{a,b,p})$ at the vertices of a regular $p$-gon in the plane and draw the edges of $C_V$ as straight-line segments. 
This gives a collection of cycles in which the internal angle at each vertex is less than $2\pi/3$ (because $p/a<6$). 
Similarly, place the vertices of $W$ at the vertices of a regular $p$-gon and draw the edges of $C_W$ as straight-line segments; 
again, the internal angle at each vertex is less than $2\pi/3$ (because $p/b<6$).
 We start with the vertices of $V$ equally spaced on the equator of $\SSS^2$ and place the vertices of $W$ so that $w_i$ coincides with $v_i$ for each $i\in[p]$. 
Move the vertices of $V$ toward the north pole until $C_V$ is drawn with all vertices on a small circle $C$ and with internal angles $2\pi/3$ at each vertex. 
Since $C_V$ is invariant under the map $v_i\mapsto v_{i+1}$, the angle formed by $C$ and the edge $v_i v_{i-a}$ equals the angle formed by $C$ and $v_i v_{i+a}$, for every $i\in[p]$.

Next, move the vertices of $W$ toward the south pole until $C_W$ is drawn with internal angles $2\pi/3$. 
We now have a configuration with $v_i$ directly ``above'' $w_i$ for each $i\in[p]$. 
At each $v_i$, the three edges $v_i v_{i-a}$, $v_i v_{i+a}$, and $v_i w_i$ meet at $2\pi/3$ angles; 
similarly, at each $w_i$, the three edges $w_i w_{i-b}$, $w_i w_{i+b}$, and $w_i v_i$ meet at $2\pi/3$ angles. 
Therefore we obtain an $\SSS^2$-immersion of $G_{a,b,p}$ in which all vertices map to distinct points. \qed

\begin{figure}[H]
\subfloat[\centering ]{{\centering
\begin{tikzpicture}[every node/.style={draw,circle,very thick}]
\begin{scope}[scale=1.1]
\draw (-1,0) arc (180:360:2cm and 0.5cm);
\draw (-1,0) arc (180:0:2cm and 0.5cm);
\draw (1,0) circle (2cm);
 \draw (-0.2,1.3) node [circle,fill, inner sep=1pt] {};
 \draw (1,1.3) node [circle,fill, inner sep=1pt] {};
 \draw (0.4,1.7) node [circle,fill, inner sep=1pt] {};
 \draw (1.5,1.7) node [circle,fill, inner sep=1pt] {};
 \draw (2.19,1.3) node [circle,fill, inner sep=1pt] {};
\draw (-0.2,1.3) arc (180:0:0.6cm and 0.2cm);
\draw (1,1.3) arc (180:0:0.6cm and 0.2cm);
\draw (0.4,1.7) arc (180:0:0.55cm and 0.2cm);
\draw (0.4,1.7) arc (270:180:0.65cm and -0.45cm);
\draw (1.5,1.7) arc (270:360:0.7cm and -0.38cm);
\draw (-0.2,-1.4) node [circle,fill, inner sep=1pt] {};
\draw (1,-1.8) node [circle,fill, inner sep=1pt] {};
\draw (0.4,-1) node [circle,fill, inner sep=1pt] {};
\draw (1.5,-1) node [circle,fill, inner sep=1pt] {};
\draw (2.19,-1.4) node [circle,fill, inner sep=1pt] {};
\draw (-0.2,-1.4) arc (180:360:1.2cm and 0.3cm);
\draw (1.5,-1) arc (-5:110:1.25cm and -0.35cm);
\draw (2.19,-1.4) arc (90:180:1.75cm and -0.35cm);
\draw (0.4,-1) arc (180:90:0.6cm and -0.8cm);
\draw (1.5,-1) arc (-20:70:0.9cm and -0.65cm);
\draw (1,-1.8)--(1,1.3) ;
\draw (2.19,-1.4) arc (90:-90:0.5cm and -1.36cm);
\draw (-0.2,-1.4) arc (90:270:0.6cm and -1.36cm);
\draw[dashed] (1.5,-1) arc (90:-90:0.5cm and -1.36cm);
\draw[dashed] (0.4,-1) arc (90:270:0.6cm and -1.36cm);
\end{scope}
\end{tikzpicture}}}
\qquad
\centering
\subfloat[\centering ]{{\includegraphics[scale=0.60]{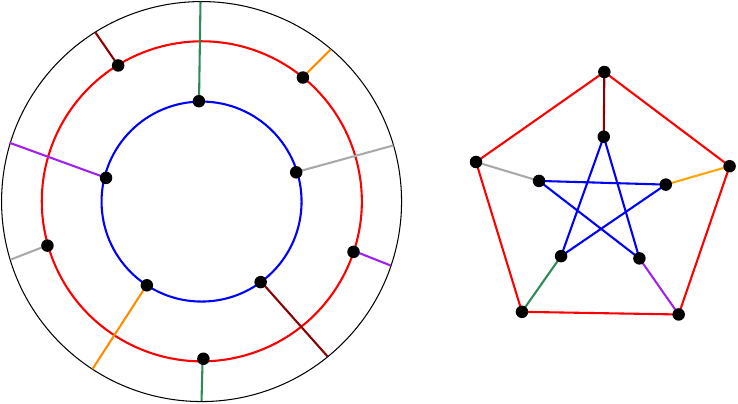} }}

\caption{Two different equiangular $\SSS^2$-immersion of the Petersen graph.}
\label{fig:Immersion_for_Petersen}
\end{figure}
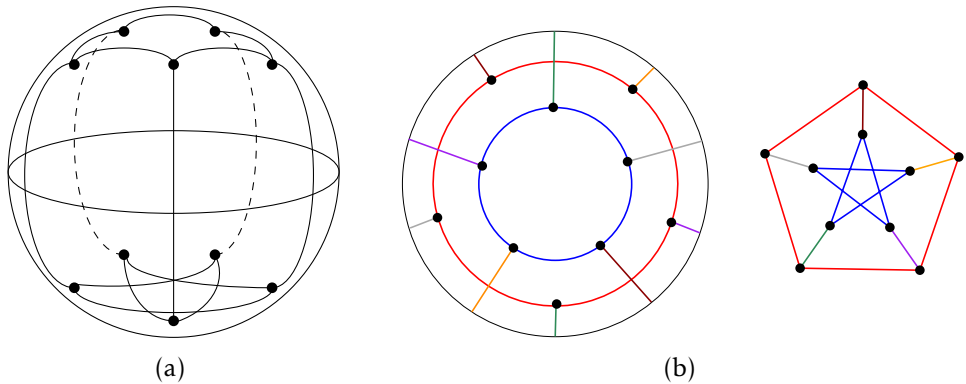




\subsection[\SSS2-immersion to One and Two Points]{$\SSS^2$-immersion to One and Two Points}

In this section, we characterise the graphs that admit an equiangular $\SSS^2$-immersion to exactly one or two points. We begin by showing that any equiangular $\SSS^2$-immersion to two points is an equiangular $\SSS^2$-immersion to antipodal points.
\begin{lem}\label{antipodal_immersion}
    Let $G$ be a cubic graph with two-points equiangular $\SSS^2$-immersion $\gamma$ that is $\gamma(V(G)) = \{p, q\}$, then the points $q$ and $p$ are antipodal.  
\end{lem}
\begin{proof}
Since $\gamma$ is a two-point $\SSS^2$-immersion, there exists at least one vertex $v \in V(G)$ such that $\gamma(v) = p$, and one of its three neighbours $x, y$ or $z$ that $\gamma$ sends to $q$, otherwise we do not have two-points equiangular $\SSS^2$-immersion. We have two cases. 

\textbf{Case $(1)$:} At least two of $\gamma(x), \gamma(y)$, and $\gamma(z)$ equals to $q$, say $x$ and $y$. Since $\gamma$ is an equiangular $\SSS^2$-immersion the arcs $\gamma(vx)$ and $\gamma(vy)$ belong to two distinct great circles that intersect in $p$ and $q$. Any two distinct great circles intersect at antipodes. Thus, $q = -p$.

\textbf{Case $(2)$:} Exactly one of $\gamma(x), \gamma(y)$, and $\gamma(z)$ is equal to $q$, say $x$. Moreover, suppose $q \neq - p$, otherwise, we are done. Let $C$ be the great circle containing $\gamma(vx)$. In this case, any point of $C \sm p$ satisfies the conditions of equiangular $\SSS^2$-immersion for the edge $e=vx$. Let $\mathdefin{G_q}$ be the union of $\gamma(e)$ for $e = uw \in E(G)$ such that $\gamma(u) = \gamma(w) = q$ and $\gamma(u) = p, \gamma(w) = q$  ; That's it $G_q$ is the drawing of the subgraph of $G$ whose  at least one of its edges' endpoints is mapped to $q$.
Let $R$ be a rotation that sends $q$ to $-p$ along the shortest arc between $q$ and $-p$. Apply $R$ to every point in $G_q$. Since $R$ is an isometry, $R$ preserves the angles between any arc from $q$ to $q$. For any arc going from $p$ to $q$, we can replace it by an arc from $p$ to $-p$ while preserving its direction along $C$. Since the arc from $p$ to $q$ did not change angles with the arcs going out of $p$, it preserves the fact that $\gamma$ is an equiangular $\SSS^2$-immersion.
\end{proof}


Now, we are ready to show that bipartite graphs are exactly the graphs that have one or two points equiangular $\SSS^2$-immersion. In order to prove the next theorem, we make use of the following well known fact about regular bipartite graphs.
\begin{proposition}\label{bipartite_pmatching}\rm\cite[Corollary 2.1.3.]{diestel2017graph}
Every $k$-regular $(k \geq 1)$ bipartite graph has a perfect matching.  
\end{proposition}

\begin{thm}\label{thm::bipartite_immersion}
Let $G$ be a cubic graph. Then the following statements are equivalent.
\begin{compactenum}
\item[(a)] $G$ has a one-point equiangular $\SSS^2$-immersion.
\item[(b)] $ G $ has a two-point equiangular $ \mathbb{S}^2 $-immersion.
\item[(c)] $G$ is bipartite.
\end{compactenum}
\end{thm}

\begin{proof}
For the direction $(a) \Longrightarrow (b)$, by applying \cref{lem:upper_half} to any $v \in V(G)$, we get an $\SSS^2$-immersion.

For the reverse direction $(b) \Longrightarrow (a)$, suppose we have a two-point $\SSS^2$-immersion $\gamma$,  by \cref{antipodal_immersion} $\gamma$ in an  ``antipodal $\SSS^2$-immersion''. Then, applying \cref{lem:upper_half} to vertices that are mapped to one of the antipodes give us a one-point $\SSS^2$-immersion. 

Now, we prove $(c) \Longleftrightarrow (b)$. For the forward direction, consider a cubic bipartite graph $G = (A \cup B, E)$. Let $\overrightarrow{G}$ be an orientation of $G$ such that every edge $e = ab \in E(G)$ with $a \in A$ and $b \in B$, $e$ is oriented from $a$ to $b$. We want to construct a two-point $\SSS^2$-immersion $\gamma$. Let $p$ be any point in $\SSS^2$. Set $\gamma(a) = p$ for every $a \in A$ and $\gamma(b) = -p$ for every $b \in B$. Furthermore, let $C_1, C_2$, and $C_3$ be three great circles with the smallest angle between any of them equals to $\pi/3$. Precisely, we have six geodesic arcs from $p$ to $-p$ along $C_1, C_2$, and $C_3$. Choose three of them at a distance $2\pi/3$, call them $C'_1, C'_2$, and $C'_3$. Note that $C'_i(0) = p$ to $C'_i(1) = -p$ for $ 1 \leq i \leq 3$.
Since $G$ is cubic bipartite graph, by \cref{bipartite_pmatching} we have three perfect matchings $M_1, M_2$, and $M_3$. For every $e_i =ab \in M_i$, let $\gamma(\overrightarrow{e_i}) = C'_i$. Now at a vertex $a \in A$, the outgoing edges each belongs to exactly one of the $M_i$, and its corresponding arc is at $2\pi/3$ from the other ones. Similarly for every vertex $b\in B$. In other words, for each vertex $a \in A$ with neighbors $b_1, b_2,$ and $b_3$, the spherical angle between $\gamma^{(1/3)}(ab_i)$ and $\gamma^{(1/3)}(ab_j)$ for $i,j\in \{1, 2, 3\}$ and $i \neq j$ is equal to $2\pi/3$ since each of $\gamma(ab_i) \subset C'_j$ for some $j \in \{1,2, 3\}$. Similar argument applies for the arcs around the vertices of $B$.

For the backward direction, suppose we have a two-point $\SSS^2$-immersion $\gamma$ of a cubic graph $G$. By \cref{antipodal_immersion}, we can assume that $\gamma(V(G)) = \{p, -p\}$ for some $p \in \SSS^2$. Note that any great circle $C$ that contains $p$ and $-p$ is contained in a plane $\pi_C \subset \mathbb{R}^3$ whose normal vector $v$ is perpendicular to the line defined by $\{p, -p\}$. The set of these normal vectors defines are subset of a circle $S \subset \SSS^2$. By similar reasoning of the proof of \cref{thm:S2}, $G$ has an $\SSS^2$-flow whose values belongs to the circle $S$. In other words, $G$ has an $\SSS^1$-flow. Since $G$ is cubic and has $\SSS^1$-flow, by \cref{cubic} $G$ is bipartite.
\end{proof}

Using \cref{cubic}, we can conclude with the following corollary.
\begin{cor}
For a cubic graph $G$, the followings are equivalent,
\begin{compactenum}
    \item Two-point $\SSS^2$ immersion.
    \item $3$-flow.
    \item $R_3$-flow.
    \item $\SSS^1$-flow.  
\end{compactenum}
\end{cor}


\subsection[Combining Cubic Graphs with S²-flows]{Combining Cubic Graphs with $\SSS^2$-flows}

Let $v$ be a vertex in a graph $G$ such that $v$ has $k$ neighbours. We define \defin{vertex splitting} at $v$ by replacing $v$ with $k$ vertices, $v_1,...,v_k$. Each vertex $v_i$ for $1 \leq i \leq k$ inherits exactly one neighbour of $v$. In other words, we replaced $v$ with $k$ leaves, each is attached to one of its neighbour.

Given two graphs $G$ and $H$ such that each contains a vertex of degree $k$ for some $k \in \mathbb{N}$. We define \defin{injecting} $H$ into $G$ at a vertex $v \in V(G)$ of degree $k$ as follow. First, we perform a vertex splitting on $v$ and on a vertex $w \in V(H)$ of the same degree. Now, we have $k$ leaves in $G$, $v_1,..., v_k$ and $k$ leaves in $H$, $w_1,..., w_k$. For each $1\leq i \leq k$, we identify $v_i$ in $G$ to the unique neighbour of $w_i$ in $H$, and identify $w_i$ to the unique neighbour of $v_i$ in $G$, we suppress the identified vertex to form an edge. We denote the new graph as $H \vartriangleright G$.

\begin{thm}\label{thm::injection_thm}
    Let $G$ and $H$ be two cubic graphs with $\SSS^2$-flows $g$ and $h$ respectively. Then, the graph $H \vartriangleright G$ created by injecting $H$ into a vertex $v$ of $G$ has an $\SSS^2$-flow.  
\end{thm}
\begin{proof}
    Let $w$ be a vertex of $H$, as $d(w) = d(v) = 3$ injecting $H$ into $G$ at $v$ is well defined. Let $a, b,$ and $c$ be the vectors in $\SSS^2$ that $g$ maps the incident edges at $v$ to. Since $g$ is a flow, we have $a + b + c = 0$. By \cref{coplanar}, $a, b,$ and $c$ are coplanar with angle $2\pi/3$ between them. Similarly, let $a',b'$ and $c'$ be the vectors in $\SSS^2$ that $h$ maps the edges around $w$ to. Note that there is bijection that maps $\{a', b', c'\}$ to $\{a, b, c\}$, precisely this bijection is a rotation of normal vectors containing $\{a', b', c'\}$ to the normal vector of the plane containing $\{a, b, c\}$. Then, it followed by another rotation to align the vector $a'$ to $a$. Denote this rotation by $\theta$. Now, we define $f$ from $E(H \vartriangleright G)$ as $f(e) = g(e)$ for $e \in E(H\vartriangleright G) \sm E(H)$, and $f(e) = \theta \circ h(e)$ for $e\in E(H)$. 
    It is clear that at every vertex of $V(G) \cap V(H \vartriangleright G)$ the flow conditions holds for $f$. For every vertex $w \in V(H) \cap V(H \vartriangleright G)$, 
    \begin{align*}
        f(e_1) + f(e_2) + f(e_3) & = \theta \circ g(e_1) + \theta \circ g(e_2) + \theta \circ g(e_3)\\
        & = \theta \circ (g(e_1) + g(e_2) + g(e_3)) \text{ (by the linearity of $\theta$)} \\
        & = \mathbf{0}    
    \end{align*}

     where $e_1, e_2, e_3$ are the edges incident to $w$.    
\end{proof}

\begin{lem}\label{lem::k_4}
    The complete graph $K^4$ has an equiangular $\SSS^2$-immersion, consequently an $\SSS^2$-flow.
\end{lem}
\begin{proof}
    We want to define an $\SSS^2$-immersion $\gamma$ for $K^4$.
    Let $v_0, v_1, v_2,$ and $v_3$ be the four vertices of $K^4$, and set $\gamma(v_0)= (0, 0, 1)$. Then, place $v_1, v_2$, and $v_3$ at equal distance from each other at the equatorial; the great circle $C_\pi$ defined by the intersection of $\SSS^2$ and the plane $\pi$ with zero $z$-coordinate. Let $\gamma(v_i)$ for $i\in \{1, 2, 3\}$ be defined as described above. Between any consecutive $\gamma(v_i)$ for $i\in \{1, 2, 3\}$ draw an arc along $C_\pi$ (this arc corresponds to the shortest possible arc), this arc between $\gamma(v_i)$ and $\gamma(v_j)$ corresponds to the edge $v_iv_j$ for $i\neq j$ and $i, j \in \{1, 2, 3\}$. Between $\gamma(v_0)$ and $\gamma(v_i)$ for $i\in \{1, 2, 3\}$ draw an arc along the unique great circle containing $\gamma(v_0)$ and $\gamma(v_i)$, call such circle $C_{v_0v_i}$. Note that $\gamma$ is not an $\SSS^2$-immersion as it is defined now; the angle between $\gamma(v_iv_0)$ and $\gamma(v_iv_j)$ for $i, j \in \{1, 2, 3\}$ is $\pi/2$, and the angle between $\gamma(v_iv_j)$ and $\gamma(v_iv_k)$ is equal to $\pi$. 

    To make $\gamma$ an $\SSS^2$-immersion, we are going to move at a constant speed each of the $\gamma(v_i)$ for $i \in \{1, 2, 3\}$ toward $-\gamma(v_0)$ along $C_{v_0v_i}$ as defined above. As this movement is going extend the arc of $\gamma(v_0v_i)$ to the new position of $\gamma(v_i)$ for $i \in \{1, 2, 3\}$. As the $\gamma(v_i)$ for $i \in \{1, 2, 3\}$ approaches $-\gamma(v_0)$, the angles between the $\gamma(v_i)$s is approaching $\pi/3$. By the intermediate value theorem, as the angle between $\gamma(v_i)$s changes between $\pi$ and $\pi/3$ there must be a position in which the $\gamma(v_i)$s attain an angle of $2\pi/3$; let such a position be the final position of each of the $\gamma(v_i)$ for $i \in \{1, 2, 3\}$. It remains to show that the angle between $\gamma(v_iv_0)$ and $\gamma(v_iv_j)$ for $i \neq j \in \{1, 2, 3\}$ is $2\pi/3$. By the symmetry of movement, as $\gamma(v_i)$ for $i \in \{1, 2, 3\}$ is moving toward $-\gamma(v_0)$ as it is described above, the angle between the arc joining $\gamma(v_0)$ to $\gamma(v_i)$ and the arcs joining $\gamma(v_i)$ to its right and left neighbors is increasing by the same amount. Therefore, the arc $\gamma(v_0v_i)$ remain an angular bisector for the angle $\gamma(v_j)\gamma(v_i)\gamma(v_k)$; consequently at the final position the value of $\gamma(v_k)\gamma(v_i)\gamma(v_0)$ equals to $(2\pi - 2\pi/3)/2 = 2\pi/3$.
\end{proof}

In the following, blowing up a vertex in a cubic graph by a triangle is to replace the vertex by a cycle of length three and give exactly one neighbour to each of its vertex. 
\begin{thm}
Let $G$ be a cubic graph. If $G$ admits an $\SSS^2$-flow, then the graph obtained by blowing up any vertex in $V(G)$ into a triangle also admits an $\SSS^2$-flow.
\end{thm}
\begin{proof}
    Blowing up a vertex $v$ in $V(G)$ into a triangle is equivalent to injecting $K^4$ into $v$. By \cref{lem::k_4} and \cref{thm::injection_thm}, the conclusion follows.
\end{proof}


\section{An Algebraic Approach of \texorpdfstring{$\SSS^2$}{S2}-Flows}\label{proof_of_algebraic_thm}

\citet{VecAndInt} proved the following:

\begin{lem}{\rm \cite[Theorem 1.10]{VecAndInt}}
If a graph $G$ admits a vector $\SSS^1$-flow with rank at most two, then $G$ admits a $3$-flow.
\end{lem} 

In this section, we generalize the previous result to $\SSS^2$ under a mild assumption. 
Let ${\pmb{\varphi}}$ be a  $\SSS^{2}$–flow on a graph $G$. Its set of flow
values is
\[
  \{\pmb{\varphi}(e)\mid e\in E(G)\}\;=\;\{\pm{\bm v}_{1},\dots,\pm{\bm v}_{b}\},
\]
where $\{{\bm v}_{1},\dots,{\bm v}_{b}\}\subseteq \SSS^{2}$ consists of
$b$ pairwise linearly independent vectors. Without loss of generality we may further assume that
$\{\pmb{\varphi}(e)\mid e\in E(G)\} = \{{\bm v}_{1},\dots,{\bm v}_{b}\}$,
because whenever an edge $e$ satisfies $\pmb{\varphi}(e)=-{\bm v}_{i}$ we simply reverse the orientation of $e$ and replace its flow value by $-\!\pmb{\varphi}(e)={\bm v}_{i}$.
For a vertex $v$, define the \defin{balanced vector} $\boldsymbol{\epsilon}(v) \in \mathbb{Z}^b$ as:
\[
\boldsymbol{\epsilon}(v) = \left( \epsilon_1(v), \dots, \epsilon_b(v) \right),
\]
where $\epsilon_i(v) = |E^+(v) \cap E_i| - |E^-(v) \cap E_i|$, and $E_i$ is the set of edges with flow value ${\bm v}_i$.
The \defin{balanced equation} at a vertex $ v $ is defined as
\begin{equation} \label{eq:balanced}
\epsilon_1(v){\bm v}_1 + \cdots + \epsilon_b(v){\bm v}_b = 0,
\end{equation}

Let $ S(\varphi) $ denote the linear subspace of $ \mathbb{R}^b $ generated by all balanced vectors of the flow $\varphi$ over $\mathbb R$. 
The \defin{rank} of this subspace, denoted by $ \operatorname{rank}(\varphi) $, is called the \defin{rank of the $ \SSS^2 $-flow} $ \varphi $.

\citet{VecAndInt} proved that $G$ admits a nowhere-zero
integer $3$-ﬂow if $G$ admits a vector $\SSS^1$-ﬂow with rank at most two. 
This result is sharp since there are examples that admit vector $\SSS^1$-ﬂows with rank at least 3, but no nowhere-zero integer $3$-ﬂows.

Let $V=\mathbb{Q}^{\,b}$ be the vector space over $\mathbb Q$. We define $\mathdefin{S_{\mathbb{Q}}(\varphi)}$ to be the linear subspace of $\mathbb{Q}^{\,b}$ generated by all balanced vectors. We call a subspace $W$ of $V$ an \defin{odd-coordinate-free} if it contains no integer vector with exactly one odd coordinate.

The following lemma is known as the finite-dimensional Riesz representation theorem.
\begin{lem}{\rm\cite[Sec.8.3 Theorem 6]{hoffman}}\label{lem:Riesz}
If $(V,\langle\,\ast,\ast\rangle)$ is a finite-dimensional inner product space over a field $\mathbb F$ and $f$ is a linear functional on $V$ (i.e $f\colon V\to \mathbb F$), then there exists a unique vector $\beta\in V$ such that $f(\alpha)=\langle \alpha,\beta\rangle$ for all $\alpha\in V$.
\end{lem}

If $V$ is a vector space over a field $\mathbb F$, a \defin{hyperspace} in $V$ is a maximal proper subspace of $V$.

\begin{lem}{\rm\cite[Sec.3.6 Theorem 19]{hoffman}}\label{lem:kernel-hyperspace}
If $f$ is a non-zero linear functional on the vector space $V$ over a field $\mathbb F$, then the kernel of $f$ is a hyperspace in $V$. 
Conversely, every hyperspace in $V$ is the kernel of a (not unique) non-zero linear functional on $V$.
\end{lem}

Although the following result is known, we provide a proof for completeness and the reader’s convenience.

\begin{lem}\label{lem:max_noit_include_1}
Let $V$ be a finite-dimensional vector space over $\mathbb F$ with a proper subspace $W$. 
If $v\notin W$, then there exists a hyperspace(maximal subspace) $M$ with
$W\subseteq M \subsetneq V$ and  $v\notin M$.
\end{lem}

\begin{proof}
Let $\{w_1,\dots,w_k\}$ be a basis of $W$. 
Since $v\notin W$, the set $\{w_1,\dots,w_k,v\}$ is linearly independent and can be extended to a basis of $V$,
say $\mathcal B=\{w_1,\dots,w_k, v, u_{k+2},\dots,u_n\}$.
Define $f\in V^*$ by prescribing its values on $\mathcal B$:
\[
f(w_i)=0 \ (1\le i\le k), \qquad f(v)=1, \qquad
f(u_j)=0 \ (k+2\le j\le n),
\]
and extend $f$ linearly to all of $V$. 
Since $f$ vanishes on $W$, we conclude that $W$ is subset of the kernel of $f$.
Furthermore by definition $f(v)=1$ which implies $v\notin \ker(f)$. 
Now it follows from \Cref{lem:kernel-hyperspace} that $\ker(f)$ is a hyperspace.
\end{proof}

Next we define the support of a vector.
Let $\xx=(x_1,\dots,x_n)$ be a vector in a vector space $V$. 
The \defin{support} of $x$ is
$\supp(\xx)\;\coloneqq \;\{\, i\in[n] \mid x_i\neq 0 \,\}$.

It is a well-known fact that in any binary linear code, either all codewords have even weight, or exactly half of them do. 

\begin{lem}\label{lem:num_odd_supp}
A $k$-dimensional binary subspace $V\subseteq\mathbb{Z}_2^n$ has either $0$ or exactly $2^{k-1}$ vectors of odd support.\qed
\end{lem}

Let $(V, \langle \ast, \ast \rangle)$ be a vector space equipped with an inner product.  
If $W$ is a subspace of $V$, it is \emph{not} necessarily true that  
$\dim(W^\perp) + \dim(W) = \dim(V)$, where  $W^\perp$ is the orthogonal complement of $W$. However, if the inner product is \emph{non-degenerate}, then the equality $\dim(W^\perp) + \dim(W) = \dim(V)$ does hold. In particular, the standard inner product on $\mathbb{Z}_2^n$ is a non-degenerate. The following lemma plays a crucial role in this section.

\begin{lem}\label{Zauber_lem}
Let $V \subseteq \ZZ_2^n$ be a linear subspace of dimension $d$ with $n - 2 \le d \le n$. If for each coordinate $i = 1, \dots ,n$, the projection map $\pi_i \colon V \to \ZZ_2$, is surjective. Then there exist two vectors ${\mathbf{x}}, {\mathbf{y}}  \in V$ such that $(\pi_i(\xx),\pi_i(\yy))\neq (0,0)$ for every $i$.
\end{lem}

\begin{proof}
Let $V^\perp = \{\,w\in \mathbb{Z}_2^{\,n} \mid w\cdot v=0\ \text{for all }v\in V\,\}$ be the orthogonal complement of $V$. Then $\dim(V^\perp)=k=n-d\in\{0,1,2\}$, since $n-2\le d\le n$. The surjectivity of each $\pi_i$ implies that no standard basis $e_i$ lies in $V^\perp$; consequently, every nonzero $w\in V^\perp$ satisfies $|\operatorname{supp}(w)|\ge 2$. Assume, for the sake of contradiction, that the lemma is false. Then we have the following: 
\begin{equation}\label{Zx}
\textit{For any pair of vectors $\xx,\yy \in V$, there exists an index $i\in [n]$ such that $x_i = y_i = 0$.}
\end{equation}
For $x\in V$ define the zero set $Z_x \;=\; \{\, i\in[n] \mid x_i=0 \,\}$.
We now claim the following:
\begin{clm}\label{clm:miracle}
for every $\xx \in V$, there exists a non-zero vector ${\bf w} \in V^\perp$ with odd support such that $\mathsf{supp}({\bf w}) \subseteq Z_{\xx}$.
\end{clm}
\begin{clmproof}
Let us fix an arbitrary vector $\xx\in V$.
Without loss of generality, we can assume that $Z_{\xx}=[|Z_{\xx}|]$.
Our hypothesis, applied to this specific $\xx$, implies that for any vector $\yy \in V$, there is an index $i \in Z_{\xx}$, where $y_i = 0$. This means that no vector in $V$ has a `1' in all coordinate positions of $Z_{\xx}$. 
Next, we consider the projection map from $V$ onto the coordinates in $Z_{\xx}$:
$$\pi_{Z_{\xx}}:  V \to \ZZ_2^{|Z_{\xx}|}\coloneqq \underbrace{\mathbb Z_2\times\cdots\times \mathbb Z_2}_{|Z_{\xx}|}$$ 
Since no vector in $V$ has a `1' in all coordinate positions of $Z_{\xx}$, we infer that  the all-ones vector $\mathbf{1} \in \ZZ_2^{|Z_{\xx}|}$ is not in the image of $\pi_{Z_{\xx}}$ and so the projection $\pi_{Z_{\xx}}$ is not a surjective map.
It follows from \Cref{lem:max_noit_include_1} that there is a hyperspace(maximal subspace) $M$ of $\ZZ_2^{|Z_{\xx}|}$ such that $\mathsf{Im}(\pi_{Z_{\xx}})\subseteq M$ and ${\bf 1}\notin M$.
By \Cref{lem:kernel-hyperspace}, let $f\colon \ZZ_2^{|Z_x|}\to \ZZ_2$ be the linear functional such that the kernel of $f$ is $M$. We now invoke \Cref{lem:Riesz} and conclude that the non-zero linear functional $f$ can be represented by the dot product with a non-zero vector ${\bf u} \in \ZZ_2^{|Z_x|}$. 
By definition of $f$, we have the following two crucial properties:
\begin{enumerate}
    \item ${\bf u} \cdot {\bf v} = 0$ for all ${\bf v} \in \mathsf{Im}(\pi_{Z_{\xx}})$, as $\mathsf{Im}(\pi_{Z_{\xx}})$ is a subset of the kernel of $f$.
    \item ${\bf u} \cdot \mathbf{1} = 1$, as $\bf 1 $ does not belong to kernel of $f$. Hence it implies that $|\mathsf{supp}({\bf u})|$ is odd.
\end{enumerate}
We extend ${\bf u}$ to a vector ${\bf w} \in \ZZ_2^n$ by setting its components to 0 for all coordinates not in $Z_{\xx}$. 
By construction, $\text{supp}({\bf w}) = \text{supp}({\bf u})$, so $\text{supp}(\bf w) \subseteq Z_{\xx}$ and $|\text{supp}({\bf w})|$ is odd. 


The first condition saying that ${\bf u} \cdot {\bf v} = 0$ for all ${\bf v} \in \mathsf{Im}(\pi_{Z_{\xx}})$.
By extending the vector ${\bf u}$ to the vector ${\bf w}$, we deduce that ${\bf w} \cdot \yy = 0$ for all $\yy \in V$, which means ${\bf w} \in V^\perp$ and the claim is proved.
\end{clmproof}

Let $ V^\perp_{odd} = \{{\bf w} \in  V^\perp \sm \{0\} \mid |\mathsf{supp}({\bf w})| \text{ is odd}\}$. 
For each ${\bf w} \in V^\perp_{odd}$, define the subspace $V_{\bf w} \coloneqq \{{\bf v} \in V \mid \text{supp}({\bf v}) \cap \text{supp}({\bf w}) = \emptyset\}$. 
We now claim the following:
\begin{clm}\label{union}
$$ V = \bigcup_{{\bf w} \in V^\perp_{odd}} V_{\bf w} $$
\end{clm}
\begin{clmproof}
It follows from \Cref{clm:miracle} that for every ${\bf v} \in V$, there exists a non-zero vector ${\bf w} \in V^\perp$ with odd support such that $\mathsf{supp}({\bf w}) \subseteq Z_{\bf v}$.
Recall that by definition, $Z_{\bf v}$ is the set of indices $i \in [n]$ for which the $i$-th coordinate of ${\bf v}$ is zero. In particular, this means that $\mathsf{supp}({\bf v}) \cap Z_{\bf v} = \emptyset$. Hence, we conclude that ${\bf v} \in V_{\bf w}$.
\end{clmproof}
By \Cref{lem:num_odd_supp}, we know that the number of vectors in $V^\perp_{odd}$ is either 0 or $2^{k-1}$, where $k=\dim(V^\perp)$. 
\begin{itemize}
    \item \textbf{Case $k=0$ ($d=n$):} $V^\perp=\{0\}$, and so $V^\perp_{odd}$ is empty and $V=\Z_2^n$.
    \Cref{union} yields a contradiction.
    \item \textbf{Case $k=1$ ($d=n-1$):} $|V^\perp_{odd}| \le 2^{1-1} = 1$. The covering becomes $V = V_{\bf w}$, where $\bf w$  lies in $V^\perp_{odd}$.
    This means that the support of each vector of $V$ has empty intersection with $\mathsf{supp}(\bf w)$.
    In other words, all vectors in $V$ are zero on $\mathsf{supp}(\bf w)$, which violates the surjectivity of the projection maps $\pi_i$ for $i \in \mathsf{supp}(\bf w)$.
    \item \textbf{Case $k=2$ ($d=n-2$):} $|V^\perp_{odd}| \le 2^{2-1} = 2$. The covering becomes $V = V_{w_1} \cup V_{w_2}$ which gives us a contradiction. 
\end{itemize}

Since every possible case arising from our initial assumption leads to a contradiction, the assumption must be false. Therefore, there must exist two vectors $\xx, \yy \in V$ such that $\text{supp}(\xx) \cup \text{supp}(\yy) = \{1, \dots, n\}$.
\end{proof}

\begin{mythm}{\ref{rank_odd_cor}}
If $G$ admits an $\SSS^2$-flow such that $S_{\mathbb{Q}}(\varphi)$ is odd-coordinate-free of rank at most $2$, then $G$ admits a $4$-flow.
\end{mythm}
\begin{proof}
Let $f$ be an $\mathbb{S}^2$-flow on a graph $G$ with rank $r \leq 2$. 
The set of flow values is $\{{\bf v}_1, \dots, {\bf v}_b\} \subseteq \mathbb{S}^2$. By definition, for any vertex $v \in V(G)$, its balanced vector $\epsilon(v) = (\epsilon_1(v), \dots, \epsilon_b(v))$ satisfies the KCL for the $\mathbb{S}^2$-flow:
$\sum_{i=1}^b \epsilon_i(v) {\bf v}_i = 0$.
Our goal is to construct a nowhere-zero 4-flow. 
By a classic result of Tutte, this is equivalent to constructing a nowhere-zero flow using the Klein four-group, $\mathbb{Z}_2 \times \mathbb{Z}_2$. 
A function $\varphi\colon E(G) \to \mathbb{Z}_2 \times \mathbb{Z}_2 \sm \{\bf 0\}$ is a nowhere-zero $\ZZ_2\times \ZZ_2$-flow if the KCL holds at every vertex. Since every element in $\ZZ_2\times \ZZ_2$ is its own inverse, the KCL simplifies to:
\[
\sum_{e \text{ incident to } v} \varphi(e) = 0 \quad \text{(in $\ZZ_2\times \ZZ_2$)}
\]

We shall define a flow $\varphi$ by assigning a value $g_i \in \mathbb Z_2\times \mathbb Z_2 \sm\{\bf 0\}$ to all edges with the flow vector ${\bf v}_i$. 
That is, if an edge $e$ has $f(e) = {\bf v}_i$, we set $\varphi(e) = g_i$. The KCL for $\varphi$ at a vertex $v$ is then:
\[
\sum_{i=1}^b \left(|E^+(v) \cap E_i| - |E^-(v) \cap E_i|\right) g_i = \sum_{i=1}^b \epsilon_i(v) g_i = 0
\]
Let $\pmb\epsilon'(v) = (\epsilon_1(v) \bmod{2}, \dots, \epsilon_b(v) \bmod{2}) \in (\mathbb{Z}_2)^b\coloneqq \underbrace{\mathbb Z_2\times\cdots\times \mathbb Z_2}_{b\textit{ times}}$. The condition becomes:
\[
\sum_{i=1}^b  {\bf \epsilon'}_i(v) g_i = 0 \quad \Leftrightarrow \quad  {\bf \epsilon'}(v) \cdot (g_1, \dots, g_b) = 0 \quad \text{(in $\mathbb Z_2\times \mathbb Z_2$)}
\]
This equation must hold for all $v \in V(G)$. Let $S' \subseteq (\mathbb{Z}_2)^b$ be the subspace spanned by all vectors $\{ {\bf \epsilon'}(v)\}$. 
We know that  $k = \dim_{\mathbb{Z}_2}(S')\le \dim _{\mathbb Q}S(f)\le 2$.
Furthermore we can represent elements of $\mathbb Z_2\times \mathbb Z_2$ as vectors in $(\mathbb{Z}_2)^2$. Let $g_i = (x_i, y_i)$, where $(x_i, y_i) \in (\mathbb{Z}_2)^2 \sm \{(0,0)\}$. The vector equation $\sum  {\bf \epsilon'}_i(v) g_i = 0$ decomposes into two independent linear equations over $\mathbb{Z}_2$:
\[
 {\pmb\epsilon'}(v) \cdot (x_1, \dots, x_b) = 0
\]
\[
 {\pmb\epsilon'}(v) \cdot (y_1, \dots, y_b) = 0
\]
These must hold for all $v$. This is equivalent to requiring the vectors ${\bf x} = (x_1, \dots, x_b)$ and ${\bf y} = (y_1, \dots, y_b)$ to lie in the orthogonal complement of $S'$, denoted $W = (S')^\perp$. The dimension of this subspace is $d = \dim(W) = b - \dim(S') = b - k$. Since $k \leq 2$, we have $d \geq b - 2$.
Our task reduces to finding two vectors ${\bf x}, {\bf y} \in W$ such that for every index $i \in \{1, \dots, b\}$, the pair $(x_i, y_i)$ is not the zero vector $(0,0)$.
In order to invoke \Cref{Zauber_lem}, we need to show that the projection onto each coordinate is surjective.
 
\begin{clm}\label{clm:proj}
We claim that the projections $\pi_1, \dots, \pi_{b}$ are surjective.
\end{clm}
\begin{clmproof}
Assume for the sake of contradiction that for some $ j \in \{1, \ldots, b\} $, the map $ \pi_j $ is not surjective. 
This implies it must be the zero map and so for any vector $ \xx \in W $, its $ j $-th coordinate is zero. 
This implies that every vector in $ W $ is orthogonal to the vector ${\bf e}_j $. 
By definition, this forces $ {\bf e}_j \in W^\perp = S' $ and so 
it can be written as a sum of some "mod 2" balanced vectors $  \pmb\epsilon'(u) $.
More precisely, we have $${{\bf e}_j}=\sum_{u\in V(G)}c_u{\pmb\epsilon'}(u)$$
We lift the equation from $\mathbb{Z}_2$ to the integers $\mathbb{Z}$.  
Let $\mathbf{z}$ be the integer vector defined by
\[
\mathbf{z} = \sum_{u \in V(G)} c_u\, \boldsymbol{\epsilon}(u).
\]
By definition, $\mathbf{z}$ is an integer linear combination of balanced vectors.  
Observe that $\mathbf{e}_j \equiv \mathbf{z} \pmod{2}$, which implies that $z_j \equiv 1 \pmod{2}$ and $z_i \equiv 0 \pmod{2}$ for every $i \neq j$ which violates the assumption ``odd-coordinate-free".
So the claim is proved.
\end{clmproof}
By \Cref{clm:proj}, we know that each projection $\pi_j$ is surjective.  
Next, we invoke \Cref{Zauber_lem} to obtain two vectors $\xx, \yy \in \mathbb{Z}^b$ such that $(\pi_i(\xx), \pi_i(\yy)) \neq (0,0)$ for every $i \in [b]$.  
We now define $g_i \coloneqq (\pi_i(\xx), \pi_i(\yy))$.  
It then follows from the discussion at the beginning of the proof that $\{g_i\}_{i \in [b]}$ forms a nowhere-zero $\mathbb{Z}_2 \times \mathbb{Z}_2$-flow.
\end{proof}

\section{\texorpdfstring{$\SSS^n$}{S-n}-Flows for \texorpdfstring{$n\geq 3$}{n>=3}}

\begin{thm}\label{H_i}
Let $G$ be a graph that can be decomposed as the union of subgraphs $H_1, \ldots, H_n$ such that each edge of $G$ belongs to exactly $l$ of these subgraphs. Suppose that each subgraph $H_i$ admits an $\SSS^{j_i}$-flow for $i = 1, \ldots, n$. Then $G$ admits an $$\SSS^{\sum_{i=1}^n j_i + n - 1}\text{-flow}$$
\end{thm}

\begin{proof}
We begin by orienting the edges of $G$ arbitrarily. For each $i \in \{1, \ldots, n\}$, let $f_i$ denote an $\SSS^{j_i}$-flow on $H_i$.
Define the flow on $G$ by assigning to each edge $e \in E$ the vector
\[
\frac{1}{\sqrt{l}}(u_1, \ldots, u_n) \in \SSS^{\sum_{i=1}^n j_i + n - 1},
\]
where each component $u_k$ is given by:
\[
u_k\coloneqq \begin{cases}
0 & \text{if } e \notin H_k, \\
f_k(e) & \text{if } e \in H_k.
\end{cases}
\]
If the orientation of $e$ in $H_k$ disagrees with its orientation in $G$, we replace $u_k$ with $-u_k$ to ensure consistency.
Since each $H_i$ admits an $\SSS^{j_i}$-flow, the resulting vector assigned to each edge satisfies the flow conservation condition at each vertex. Therefore, $G$ admits an $\SSS^{\sum_{i=1}^n j_i + n - 1}$-flow.
\end{proof}

As an immediate consequence of~\Cref{H_i}, we obtain the following corollary.
If the graph $G$ is bridgeless, then, as shown in , there exists a collection of seven even subgraphs such that each edge of $G$ lies in exactly four of them. It follows that every bridgeless graph admits a flow with values in the 6-dimensional unit sphere 
In other words, if $G$ is bridgeless, then it admits an $\SSS^6$-flow, as shown in \cite{thom2}.

\begin{cor}
If $G$ is a bridgeless graph, then the following holds:
\begin{compactenum}
\item $G$ admits an $\SSS^6$-flow {\rm\cite{ber}}.
\item Under the Berge–Fulkerson conjecture (\cite{seymour1979multi,fulkerson1971blocking}),  every bridgeless cubic graph admits an $\SSS^5$-flow.
\item Under the Celmins and Preissmann conjecture, every bridgeless graph admits an $\SSS^4$-flow.
\end{compactenum}
\end{cor}

\begin{proof}
$ $
\begin{compactenum}
\item By \cite[Theorem 6.6.1]{diestel2017graph}, every bridgeless graph admits a nowhere-zero $\ZZ_6$-flow $\varphi$.
Define the subgraph $H_1 \coloneqq \{e \in E \mid \varphi(e) \not\equiv 0 \pmod{2}\}$. 
Then $H_1$ admits a $\ZZ_2$-flow, and by~\Cref{obs:S0}, it follows that $H_1$ admits an $\SSS^0$-flow and is a disjoint union of cycles.
Next, define $H_2 \coloneqq \{e \in E \mid \varphi(e) \not\equiv 0 \pmod{3}\}$. 
Modify the flow $\varphi$ by adding $1$ to each cycle in $H_2$ in the clockwise direction to obtain a new nowhere-zero flow $\psi$. 
Define $H_3 \coloneqq \{e \in E \mid \psi(e) \not\equiv 0 \pmod{3}\}$.
Repeat this process once more, this time adding $1$ to each cycle in the counterclockwise direction, and define $H_4 = \{e \in E \mid \psi(e) \not\equiv 0 \pmod{3}\}$.
Observe that $H_2$, $H_3$, and $H_4$ each admit a $\ZZ_3$-flow. 
By~\Cref{cubic}, it follows that each of these subgraphs also admits an $\SSS^1$-flow.
It is straightforward to verify that every edge of $G$ lies in exactly three of the subgraphs $H_1, H_2, H_3, H_4$. 
Therefore, applying~\Cref{H_i}, we conclude that $G$ admits an $\SSS^6$-flow, as claimed.

\item The conjecture stating that $G$ has six cycles such that every edge appears in exactly four of them. 

\item The conjecture stating that five cycles covering each edge exactly twice.
\end{compactenum}
\end{proof}

\section{Further Research}

In this paper, we studied unit vector flows with respect to the Euclidean norm for finite graphs. 
A natural question is what happens when we switch from the Euclidean norm to other $L_p$ norms on $\mathbb{R}^n$, or more generally, to an arbitrary norm on $\mathbb{R}^n$. One may even take a further step and investigate unit vector flows over other vector spaces, such as the $p$-adic integers.

Furthermore flow theory has also been developed for infinite graphs, see \cite{MR3621723} and using those tools we can extend the results of this paper to infinite graphs.
Last but not least, we close the paper with the following question:

\begin{enumerate}
    \item Can we drop the condition ``odd-free-coordinate" in \Cref{rank_odd_cor}?
\end{enumerate}



\bibliographystyle{plainurlnat}
\bibliography{uvf}

\end{document}